\documentclass[12pt, a4paper]{amsart}
\usepackage[hmargin=32mm, vmargin=27mm, includefoot, twoside]{geometry}
\usepackage[bookmarksopen=true]{hyperref}

\usepackage[english]{babel}
\usepackage{amsmath,amssymb}
\usepackage{mathrsfs}
\usepackage{dsfont}
\usepackage{amscd}
\usepackage{lmodern}
\usepackage{caption}
\usepackage{graphicx}
\usepackage{color}

\newtheorem{thmintro}{Theorem}

\newtheorem{corintro}[thmintro]{Corollary}

\newtheorem{theorem}{Theorem}[section]

\newtheorem{lemma}[theorem]{Lemma}
\newtheorem{prop}[theorem]{Proposition}
\theoremstyle{definition}

\newtheorem{remark}[theorem]{Remark}
\newtheorem{example}[theorem]{Example}

\newcommand{\NN}{\mathbb{N}}

\newcommand{\RR}{\mathbb{R}}

\newcommand{\AAA}{\mathcal{A}}

\newcommand{\CCC}{\mathcal{C}}

\newcommand{\WW}{\mathcal{W}}

\newcommand{\la}{\langle}
\newcommand{\ra}{\rangle}

\newcommand{\co}{\colon\thinspace}

\DeclareMathOperator{\Ch}{Ch}
\DeclareMathOperator{\Conv}{Conv}
\DeclareMathOperator{\CGR}{CGR}

\DeclareMathOperator{\Geo}{Geo}
\DeclareMathOperator{\proj}{proj}
\DeclareMathOperator{\Res}{Res}

\DeclareMathOperator{\St}{St}

\begin{document}

\renewcommand{\proofname}{{\bf Proof}}

\title{On geodesic ray bundles in buildings}

\author[T.~Marquis]{Timoth\'ee \textsc{Marquis}$^*$}
\address{UCL, 1348 Louvain-la-Neuve, Belgium}
\email{timothee.marquis@uclouvain.be}
\thanks{$^*$F.R.S.-FNRS Postdoctoral Researcher}

\begin{abstract}
Let $X$ be a building, identified with its Davis realisation. In this paper, we provide for each $x\in X$ and each $\eta$ in the visual boundary $\partial X$ of $X$ a description of the geodesic ray bundle $\Geo(x,\eta)$, namely, of the union of all combinatorial geodesic rays (corresponding to infinite minimal galleries in the chamber graph of $X$) starting from $x$ and pointing towards $\eta$. When $X$ is locally finite and hyperbolic, we show that the symmetric difference between $\Geo(x,\eta)$ and $\Geo(y,\eta)$ is always finite, for $x,y\in X$ and $\eta\in\partial X$. This gives a positive answer to a question of Huang, Sabok and Shinko in the setting of buildings. Combining their results with a construction of Bourdon, we obtain examples of hyperbolic groups $G$ with Kazhdan's property (T) such that the $G$-action on its Gromov boundary is hyperfinite.
\end{abstract}

\maketitle

\section{Introduction}

This paper is motivated by a question of Huang, Sabok and Shinko (\cite[Question~1.5]{HSS17}), asking whether in a proper and cocompact hyperbolic space $X$, the symmetric difference between two geodesic ray bundles pointing in the same direction is always finite (see below for precise definitions). 

This is motivated by the study of Borel equivalence relations for the action of a hyperbolic group $G$ on its Gromov boundary: the authors of \cite{HSS17} give a positive answer to the above question when $X$ is a CAT(0) cube complex, and deduce that if $G$ is a hyperbolic cubulated group (namely, if $G$ acts properly and cocompactly on a CAT(0) cube complex), then the $G$-action on its Gromov boundary $\partial G$ is hyperfinite, that is, it induces a hyperfinite equivalence relation (\cite[Corollary~1.2]{HSS17}). 

As it turns out, the answer to \cite[Question~1.5]{HSS17} is ``no'' in full generality: counter-examples are constructed in \cite{Tou17}. The purpose of this paper is to give a positive answer to this question when $X$ is a hyperbolic locally finite building. We underline that the class of groups acting properly and cocompactly on hyperbolic locally finite buildings includes groups with Kazhdan's property (T), and is thus significantly different from the class of cubulated hyperbolic groups considered in \cite{HSS17} (see the fixed point theorem \cite{NR97}). We now give a precise statement of our main result.

By a classical result of M.~Davis (\cite{Davis}), any building $\Delta$ can be realised as a complete CAT(0) metric space $(X,d)$, which can be viewed as a subcomplex of the barycentric subdivision of the standard geometric realisation of $\Delta$. Let $X^{(0)}\subseteq X$ denote the set of barycenters of chambers of $X$, that is, $X^{(0)}$ is the $0$-skeleton of the chamber graph of $\Delta$. The \emph{boundary} $\partial X$ of $X$ is the set of equivalence classes of asymptotic geodesic rays in $X$ (see Section~\ref{section:Preliminaries} for precise definitions). We denote, for each $x\in X^{(0)}$ and $\eta\in\partial X$, by $\Geo(x,\eta)\subseteq X^{(0)}$ the union of all combinatorial geodesic rays $\Gamma=(x_n)_{n\in\NN}\subseteq X^{(0)}$ (i.e. if $x_n$ is the barycenter of the chamber $C_n$, then $(C_n)_{n\in\NN}$ in an infinite minimal gallery in $\Delta$) starting at $x_0=x$ and pointing towards $\eta$, in the sense that $\Gamma$ is contained in a tubular neighbourhood of some geodesic ray towards $\eta$. The sets $\Geo(x,\eta)$ are called \emph{geodesic ray bundles}. 

In this paper, we give a description of geodesic ray bundles in arbitrary buildings (see Section~\ref{subsection:CBOAB} and Proposition~\ref{prop:basic_description_Geo}). When the building $X$ is (Gromov) hyperbolic and locally finite, we deduce from this description the following theorem.

\begin{thmintro}\label{thmintro:main}
Let $X$ be a locally finite hyperbolic building. Let $x,y\in X^{(0)}$ and let $\eta\in\partial X$. Then the symmetric difference of $\Geo(x,\eta)$ and $\Geo(y,\eta)$ is finite.
\end{thmintro}

As an immediate consequence of Theorem~\ref{thmintro:main} and of \cite[Theorem~1.4]{HSS17}, we deduce the following corollary.

\begin{corintro}\label{corintro:Bourdon}
Let $G$ be a group acting cocompactly on a locally finite hyperbolic building $X$, and assume that $G$ acts freely on the chambers of $X$. Then the natural action of $G$ on its Gromov boundary is hyperfinite.
\end{corintro}

In \cite[\S 1.5.3]{Bou00} (see also \cite{Sw01}), M.~Bourdon constructs a family of groups $G$ with property (T) acting cocompactly on some hyperbolic building $X$. These groups $G$ are defined as fundamental groups of some \emph{complexes of groups} (a standard reference to this topic is \cite{BHCAT0}), and it follows straightaway from the form of the complexes of groups involved that $G$ acts freely on the set of chambers of $X$ and that $X$ is locally finite. Another example of such a group with an explicit short presentation also recently appeared in \cite{Cap17}.

In particular, Corollary~\ref{corintro:Bourdon} yields examples of hyperbolic groups with property (T) whose boundary action is hyperfinite.

\begin{corintro}\label{corintro:T}
There exist (infinite) hyperbolic groups $G$ with property (T) such that the $G$-action on its Gromov boundary is hyperfinite.
\end{corintro}

Note that any group with property (T) that acts on a CAT(0) cube complex has a global fixed point (\cite{NR97}). In particular, Theorem~\ref{thmintro:main} covers situations that are not covered by \cite{HSS17} (see also the last paragraph in the introduction of \cite{HSS17}).

\smallskip 

\noindent
{\bf Acknowledgement.}
I would like to thank Pierre-Emmanuel Caprace for bringing \cite[Question~1.5]{HSS17} to my attention and for suggesting to explore it in the context of buildings. I would also like to thank the anonymous referee for his/her precious comments.

\section{Preliminaries}\label{section:Preliminaries}

\subsection{CAT(0)-spaces and Gromov hyperbolic spaces}\label{subsection:CAT0}
The standard reference for this paragraph is \cite{BHCAT0}.

Let $(X,d)$ be a complete CAT(0)-space, namely, a complete geodesic metric space in which every triangle $\Delta$ is at least as thin as the corresponding triangle $\Delta'$ in Euclidean space $\mathbb{E}$ with same side lengths, in the sense that any two points $x,y$ of $\Delta$ are at distance at most $d_{\mathbb{E}}(x',y')$ from one another, where $x',y'$ are the points on $\Delta'$ corresponding to $x,y$ respectively, and $d_{\mathbb{E}}$ is the Euclidean distance on $\mathbb{E}$.

Given two points $x,y\in X$, there is a unique geodesic segment from $x$ to $y$, which we denote $[x,y]$. A {\bf geodesic ray} based at $x\in X$ is an isometry $r\co \RR_{\geq 0}\to X$ with $r(0)=x$. Two geodesic rays $r,r'$ are called {\bf asymptotic} if $\sup_{t\in\RR_{\geq 0}}d(r(t),r'(t))<\infty$. Equivalently, identifying $r,r'$ with their image in $X$, they are asymptotic if they are at bounded {\bf Hausdorff distance} from one another, that is, if $r$ (resp. $r'$) is contained in a tubular neighbourhood of $r'$ (resp. $r$). We recall that a {\bf tubular neighbourhood} of a subset $S$ of $X$ is just an $\epsilon$-neighbourhood of $S$ for some $\epsilon>0$. The {\bf boundary} of $X$, denoted $\partial X$, is the set of equivalence classes $[r]$ of geodesic rays $r\subseteq X$, where two geodesic rays are equivalent if they are asymptotic. We then say that the geodesic ray $r$ {\bf points towards} $\eta:=[r]\in\partial X$. For each $x\in X$ and $\eta\in\partial X$, there is a unique geodesic ray starting at $x$ and pointing towards $\eta$, which we denote $[x,\eta)$.

The space $X$ is called {\bf (Gromov) hyperbolic} if there is some $\delta>0$ such that each triangle $\Delta$ in $X$ is $\delta$-slim, in the sense that each side of $\Delta$ is contained in a $\delta$-neighbourhood of the other two ($X$ is then also called {\bf $\delta$-hyperbolic}). Hyperbolic spaces can be thought of as fattened versions of trees, and their behavior is somehow opposite to that of a Euclidean space: if the CAT(0) space $X$ is {\bf proper} (every closed ball of $X$ is compact) and {\bf cocompact} (there is some compact subset $C\subseteq X$ such that $\mathrm{Isom}(X).C=X$), then $X$ is hyperbolic if and only if it does not contain a subspace isometric to the Euclidean plane.

There is a notion of {\bf Gromov boundary} of a hyperbolic space; in the context of CAT(0) spaces, it coincides with the boundary defined above, endowed with the cone topology (see \cite[II.8]{BHCAT0}).

\subsection{Buildings}\label{subsection:Buildings}
The standard reference for this paragraph is \cite{BrownAbr}.

Let $\Delta$ be a building, viewed as a simplicial complex (see \cite[Chapter~4]{BrownAbr}). Let $\Ch(\Delta)$ denote the set of {\bf chambers} (i.e. maximal simplices) of $\Delta$. A {\bf panel} is a codimension $1$ simplex of $\Delta$. Two chambers are {\bf adjacent} if they share a common panel. A {\bf gallery} between two chambers $C,D\in\Ch(\Delta)$ is a sequence $\Gamma=(C_0=C,C_1,\dots,C_k=D)$ of chambers such that $C_{i-1}$ and $C_i$ are distinct and adjacent for each $i=1,\dots,k$. The integer $k$ is called the {\bf length} of $\Gamma$. If $\Gamma$ is a gallery of minimal length between $C$ and $D$, it is called a {\bf minimal gallery} and its length is denoted $d_{\Ch}(C,D)$. The map $d_{\Ch}\co\Ch(\Delta)\times\Ch(\Delta)\to\NN$ is then a metric, called the {\bf chamber distance} on $\Delta$. An infinite sequence $\Gamma=(C_i)_{i\in\NN}$ of chambers is called a {\bf minimal gallery} if $(C_0,\dots,C_n)$ is a minimal gallery for each $n\in\NN$. Any such $\Gamma$ is contained in an apartment $A$ of $\Delta$.

Let $A$ be an apartment of $\Delta$ and let $C,D\in\Ch(A)$ be distinct adjacent chambers in $A$. Then no chamber of $A$ is at equal (chamber) distance from $C$ and $D$; this yields a partition $\Ch(A)=\Phi(C,D)\dot{\cup}\Phi(D,C)$, where $\Phi(C,D)$ is the set of chambers that are closer to $C$ than to $D$. The subcomplexes of $A$ with underlying chamber sets $\Phi(C,D)$ and $\Phi(D,C)$ are called {\bf half-spaces} or {\bf roots}, and their intersection is called the {\bf wall} separating $C$ from $D$. If $m$ is a wall delimiting the half-spaces $\Phi_+,\Phi_-$ of $A$, we say that two subsets $S_+\subseteq\Phi_+$ and $S_-\subseteq\Phi_-$ are {\bf separated} by $m$.

A gallery $\Gamma=(C_0,\dots,C_k)$ (resp. $\Gamma=(C_i)_{i\in\NN}$) contained in an apartment $A$ is said to {\bf cross} a wall $m$ of $A$ if $m$ is the wall separating $C_{i-1}$ from $C_i$ for some $i\in\{1,\dots,k\}$ (resp. $i\in\NN_{>0}$). The gallery $\Gamma\subseteq A$ is then minimal if and only if it crosses each wall of $A$ at most once. Moreover, if $C,D\in\Ch(A)$, then the set of walls crossed by a minimal gallery $\Gamma$ from $C$ to $D$ depends only on $C,D$, i.e. it is independent of the choice of $\Gamma$. 

If $A$ is an apartment of $\Delta$ and $C\in\Ch(A)$, there is a simplicial map $\rho_{A,C}\co \Delta\to A$, called the {\bf retraction onto $A$ centered at $C$}, with the following properties: $\rho_{A,C}$ is the identity on $A$ and its restriction to any apartment $A'$ containing $C$ is an isomorphism, with inverse $\rho_{A',C}|_{A}\co A\to A'$ (in particular, $\rho_{A,C}$ preserves the minimal galleries from $C$). Moreover, $\rho_{A,C}$ does not increase the distance: $d_{\Ch}(\rho_{A,C}(D),\rho_{A,C}(E))\leq d_{\Ch}(D,E)$ for all $D,E\in\Ch(\Delta)$.

The set of all panels of $\Delta$ is denoted $\Res_1(\Delta)$. The {\bf star} of a panel $\sigma\in\Res_1(\Delta)$, denoted $\St(\sigma)$, is the set of chambers containing $\sigma$. For any panel $\sigma\in\Res_1(\Delta)$ and any chamber $C\in\Ch(\Delta)$, there is a unique chamber $C'$ in $\St(\sigma)$ minimising the gallery distance from $C$ to $\St(\sigma)$; it is called the {\bf projection of $C$ on $\sigma$} and is denoted $\proj_{\sigma}(C):=C'$. It has the following {\bf gate property}: $d_{\Ch}(C,D)=d_{\Ch}(C,C')+d_{\Ch}(C',D)$ for all $D\in \St(\sigma)$.

The building $\Delta$ is called {\bf locally finite} if $\St(\sigma)$ is a finite set of chambers for each $\sigma\in\Res_1(\Delta)$.

\subsection{Davis realisation of a building}\label{subsection:Davis_realisation}
The standard reference for this paragraph is \cite[Chapter~12]{BrownAbr} (see also \cite{Davis}).

Let $\Delta$ be a building. Then $\Delta$ admits a CAT(0)-realisation $(X,d)$, called the {\bf Davis realisation} of $\Delta$, which is a complete CAT(0) space. It can be viewed as a subcomplex of the barycentric subdivision of the standard geometric realisation of $\Delta$, and contains the barycenter of each chamber and panel of $\Delta$. In the sequel, we will often identify $\Delta$ with its Davis realisation $X$, and all related notions (apartment, chamber, panel, gallery, wall,\dots) with their realisation in $X$ (viewed as closed subspaces of $X$). 

We set 
$$X^{(0)}:=\{x_C \ | \ C\in\Ch(\Delta)=\Ch(X)\}\subseteq X,$$ where $x_C\in X$ denotes the barycenter of the chamber $C$. If $A$ is an apartment of $X$, we also set $A^{(0)}:=A\cap X^{(0)}$. 

A {\bf combinatorial path} between $x_C,x_D\in X^{(0)}$ is a piecewise geodesic path $\Gamma\subseteq X$ which is the union of the geodesic segments $[x_{C_{i-1}},x_{C_i}]$ ($i=1,\dots,k$) for some gallery $(C=C_0,C_1,\dots,C_{k}=D)$; we then write $\Gamma=(x_{C_i})_{0\leq i\leq k}$. Thus combinatorial paths in $X$ are in bijection with galleries in $\Delta$. A {\bf combinatorial geodesic} is a combinatorial path corresponding to a minimal gallery in $\Delta$. One defines similarly infinite combinatorial paths and {\bf combinatorial geodesic rays} (abbreviated CGR) by replacing galleries in $\Delta$ with infinite galleries. If $\eta\in\partial X$, then a {\bf combinatorial geodesic ray from $x\in X^{(0)}$ to $\eta$} is a combinatorial geodesic ray starting at $x$ and at bounded Hausdorff distance from some (any) geodesic ray pointing towards $\eta$. We denote by $\CGR(x,\eta)$ the set of CGR from $x\in X^{(0)}$ to $\eta\in\partial X$. If $\Gamma_{xy}$ is a combinatorial geodesic from some $x\in X^{(0)}$ to some $y\in X^{(0)}$, and if $\Gamma_{yz}$ is a combinatorial geodesic (resp. ray) from $y$ to some $z\in X^{(0)}$ (resp. $z\in\partial X$), we denote by $\Gamma_{xy}\cdot\Gamma_{yz}$ the combinatorial path obtained as the concatenation of $\Gamma_{xy}$ and $\Gamma_{yz}$.

Each geodesic segment (resp. geodesic ray) of $X$ is contained in some minimal gallery, and hence also in some apartment $A$ of $X$; in particular, $\partial X$ is covered by the boundaries $\partial A$ of all apartments $A$ of $X$. Conversely, the uniqueness of geodesic rays implies that if $x\in A$ and $\eta\in\partial A$ for some apartment $A$ of $X$, then $[x,\eta)\subseteq A$. Of course, any combinatorial geodesic ray is also contained in some apartment of $X$. For every apartment $A$ and $x=x_C\in A^{(0)}$, the retraction $\rho_{A,C}\co\Delta\to A$ induces a retraction $\rho_{A,x}\co X\to A$ with the same properties as the ones described in \S\ref{subsection:Buildings}. Moreover, $d(\rho_{A,x}(y),\rho_{A,x}(z))\leq d(y,z)$ for all $y,z\in X$, with equality if $y$ belongs to the closed chamber $C\subseteq X$.

Let $A$ be an apartment of $X$. Here are a few important properties of walls in $A$, which can be found in \cite{Nos11} (see also \cite{singlepoint}). A wall $m$ of $A$ that intersects a geodesic (resp. geodesic ray) in more than one point entirely contains that geodesic (resp. geodesic ray); in particular, $m$ is convex. The subset $A\setminus m$ of $A$ has two connected components (the open half-spaces corresponding to $m$), and those components are convex. As we saw in \S\ref{subsection:Buildings}, a combinatorial path $\Gamma=(x_{C_i})_{0\leq i\leq k}$ (resp. $\Gamma=(x_{C_i})_{i\in\NN}$) contained in $A$ is a combinatorial geodesic (resp. a CGR) if and only if it crosses each wall of $A$ at most once.

Note that $X$ is a proper CAT(0) space if and only if it is locally finite. The building $X$ is called {\bf hyperbolic} if it is hyperbolic in the sense of \S\ref{subsection:CAT0} when equipped with the CAT(0) metric $d$. Equivalently, $X$ is hyperbolic if and only if $(A,d)$ is hyperbolic for some (resp. for each) apartment $A$ of $X$, as readily follows from the properties of retractions onto apartments.
 Note that Moussong gave a characterisation of the hyperbolicity of $X$ in terms of the type $(W,S)$ of $\Delta$ (see \cite[Theorem~17.1]{Mou88}): $X$ is hyperbolic if and only if $(W_J:=\la J\ra,J)$ is not an affine Coxeter system whenever $|J|\geq 3$, and there is no pair of disjoint subsets $I,J\subseteq S$ such that $W_{I}$ and $W_J$ are infinite and commute. The only fact that we will need about hyperbolic buildings, however, is the following.

\begin{lemma}\label{lemma:hyperbolic_basic_prop}
Assume that the building $X$ is hyperbolic. Then there is a constant $K>0$ such that for any $x\in X^{(0)}$ and $\eta\in\partial X$, any $\Gamma\in\CGR(x,\eta)$ is contained in a $K$-neighbourhood of $[x,\eta)$. 
\end{lemma}
\begin{proof}
By \cite[Proposition~I.7.31]{BHCAT0}, combinatorial geodesics are quasi-geodesics (for the CAT(0) metric $d$), so that the lemma follows from \cite[Theorem~III.1.7]{BHCAT0}.
\end{proof}

Here is also a basic useful fact about combinatorial geodesic rays.

\begin{lemma}\label{lemma:preparation}
Let $x\in X^{(0)}$, and let $\Gamma=(x_n)_{n\in\NN}$ be a CGR. Then there exists some $k\in\NN$ such that $\Gamma_{xx_k}\cdot (x_n)_{n\geq k}$ is a CGR for any combinatorial geodesic $\Gamma_{xx_k}$ from $x$ to $x_k$.
\end{lemma}
\begin{proof}
Reasoning inductively, we may assume that $x$ and $x_0$ are adjacent. If $(x,x_0)\cdot\Gamma$ is a CGR, the claim is clear with $k=0$. Otherwise, there is some $m\geq 1$ such that the combinatorial path $(x,x_{0},\dots,x_{m})$ is not a combinatorial geodesic, so that $d_{\Ch}(x,x_{m})\leq m$. Let $\Gamma_{xx_{m}}$ be any combinatorial geodesic from $x$ to $x_{m}$. If $\Gamma_{xx_{m}}\cdot (x_{n})_{n\geq m}$ is a CGR, we are done with $k=m$. Otherwise, there is some $k>m$ such that the combinatorial path $\Gamma_{xx_{m}}\cdot (x_{m},\dots,x_{k})$ is not a combinatorial geodesic, so that $d_{\Ch}(x,x_{k})\leq k-1$. Let $\Gamma_{xx_{k}}$ be any combinatorial geodesic from $x$ to $x_{k}$. We claim that $\Gamma_{xx_{k}}\cdot (x_{n})_{n\geq k}$ is a CGR, yielding the lemma. Indeed, otherwise there is some $\ell>k$ such that the combinatorial path $\Gamma_{xx_k}\cdot (x_{n})_{k\leq n\leq \ell}$ is not a combinatorial geodesic, and hence $$\ell-1=d_{\Ch}(x_0,x_{\ell})-1\leq d_{\Ch}(x,x_{\ell})< d_{\Ch}(x,x_{k})+d_{\Ch}(x_{k},x_{\ell})\leq k-1+\ell-k=\ell-1,$$ a contradiction.
\end{proof}

\section{Combinatorial bordification of a building}\label{subsection:CBOAB}
In this section, we recall the notion of combinatorial bordification of a building introduced in \cite{CL11}, and relate it to the notions introduced in Section~\ref{section:Preliminaries}.

Let $\Delta$ be a building, as in \S\ref{subsection:Buildings}. Recall that for each panel $\sigma\in \Res_1(\Delta)$, we have a projection map $\proj_{\sigma}\co\Ch(\Delta)\to \St(\sigma)\subseteq\Ch(\Delta)$ associating to each chamber $C$ the unique chamber of $\St(\sigma)$ closest to $C$. This defines an injective map
$$\pi_{\Ch}\co\Ch(\Delta)\to\prod_{\sigma\in\Res_1(\Delta)}\St(\sigma): C\mapsto \big(\sigma\mapsto \proj_{\sigma}(C)\big).$$
We endow $\prod_{\sigma\in\Res_1(\Delta)}\St(\sigma)$ with the product topology, where each star $\St(\sigma)$ is a discrete set of chambers. The {\bf (minimal) combinatorial bordification} of $\Delta$ is then defined as the closure
$$\CCC_1(\Delta):=\overline{\pi_{\Ch}(\Ch(\Delta))}\subseteq \prod_{\sigma\in\Res_1(\Delta)}\St(\sigma).$$

Since $\pi_{\Ch}$ is injective, we may identify $\Ch(\Delta)$ with a subset of $\CCC_1(\Delta)$, and it thus makes sense to say that a sequence of chambers $(C_n)_{n\in\NN}$ {\bf converges} to some $\xi\co \Res_1(\Delta)\to\Ch(\Delta)$ in $\CCC_1(\Delta)$. If $\Delta$ is reduced to a single apartment $A$, this notion of convergence is transparent: $(C_n)_{n\in\NN}$ converges in $\CCC_1(A)$ if and only if for every wall $m$ of $A$, the sequence $(C_n)_{n\in\NN}$ eventually remains on the same side of $m$. On the other hand, back to a general $\Delta$, one can identify $\CCC_1(A)$ ($A$ an apartment) with the subset of $\CCC_1(\Delta)$ consisting of limits of sequences of chambers in $A$, and in fact (see \cite[Proposition~2.4]{CL11})
$$\CCC_1(\Delta)=\bigcup_{\textrm{$A$ apartment}}\CCC_1(A).$$ 

Let $C\in\Ch(\Delta)$, and let $(C_n)_{n\in\NN}$ be a sequence of chambers converging to some $\xi\in\CCC_1(\Delta)$. We define the {\bf combinatorial sector} based at $C$ and pointing towards $\xi$ as
$$Q(C,\xi):=\bigcup_{k\geq 0}\bigcap_{n\geq k}\Conv(C,C_n),$$
where $\Conv(C,C_n)$ denotes the union of all minimal galleries from $C$ to $C_n$. Then $Q(C,\xi)$ indeed only depends on $C$ and $\xi$ (and not on the choice of sequence $(C_n)_{n\in\NN}$ converging to $\xi$), and is contained in some apartment. Note also that if $C'\in\Ch(\Delta)$ is contained in $Q(C,\xi)$, then
$$Q(C',\xi)\subseteq Q(C,\xi).$$

\begin{example}\label{example:A21}
Let $X$ be a building of type $\widetilde{A}_2$. The apartments of $X$ are then Euclidean planes tesselated by congruent equilateral triangles. If $A$ is an apartment of $X$, its bordification $\CCC_1(A)$ consists of $6$ ``lines of points'' and $6$ ``isolated points'' (see \cite[Example~2.6]{CL11}), which can be seen as follows. Let $x\in A^{(0)}$ and $\eta\in\partial A$.

If the direction $\eta$ is non-singular, in the sense that $[x,\eta)$ is not contained in a tubular neighbourhood of any wall of $A$, then for any $\Gamma=(x_n)_{n\in\NN}\in\CGR(x,\eta)$ contained in $A$, the sequence of (barycenters of) chambers $(x_n)_{n\in\NN}$ converges in $\CCC_1(A)$, to a unique $\xi\in\CCC_1(A)$. The sector $Q(x,\xi)$ in $A$ is shown on Figure~\ref{figure:nonsingular}.

If $\eta$ is singular, that is, if $[x,\eta)$ is contained in a tubular neighbourhood of some wall of $A$, then the set of $\xi\in\CCC_1(A)$ obtained as above as the limit of some $\Gamma\in\CGR(x,\eta)$ are the vertices of some simplicial line ``at infinity'' (see the dashed line on Figure~\ref{figure:singular}), and the combinatorial sectors $Q(x,\xi)$ for $\xi$ on this line are represented on Figure~\ref{figure:singular}. 
\end{example}

\begin{figure}
\centering
  \includegraphics[trim = 20mm 5mm 40mm 0mm, clip, width=8cm]{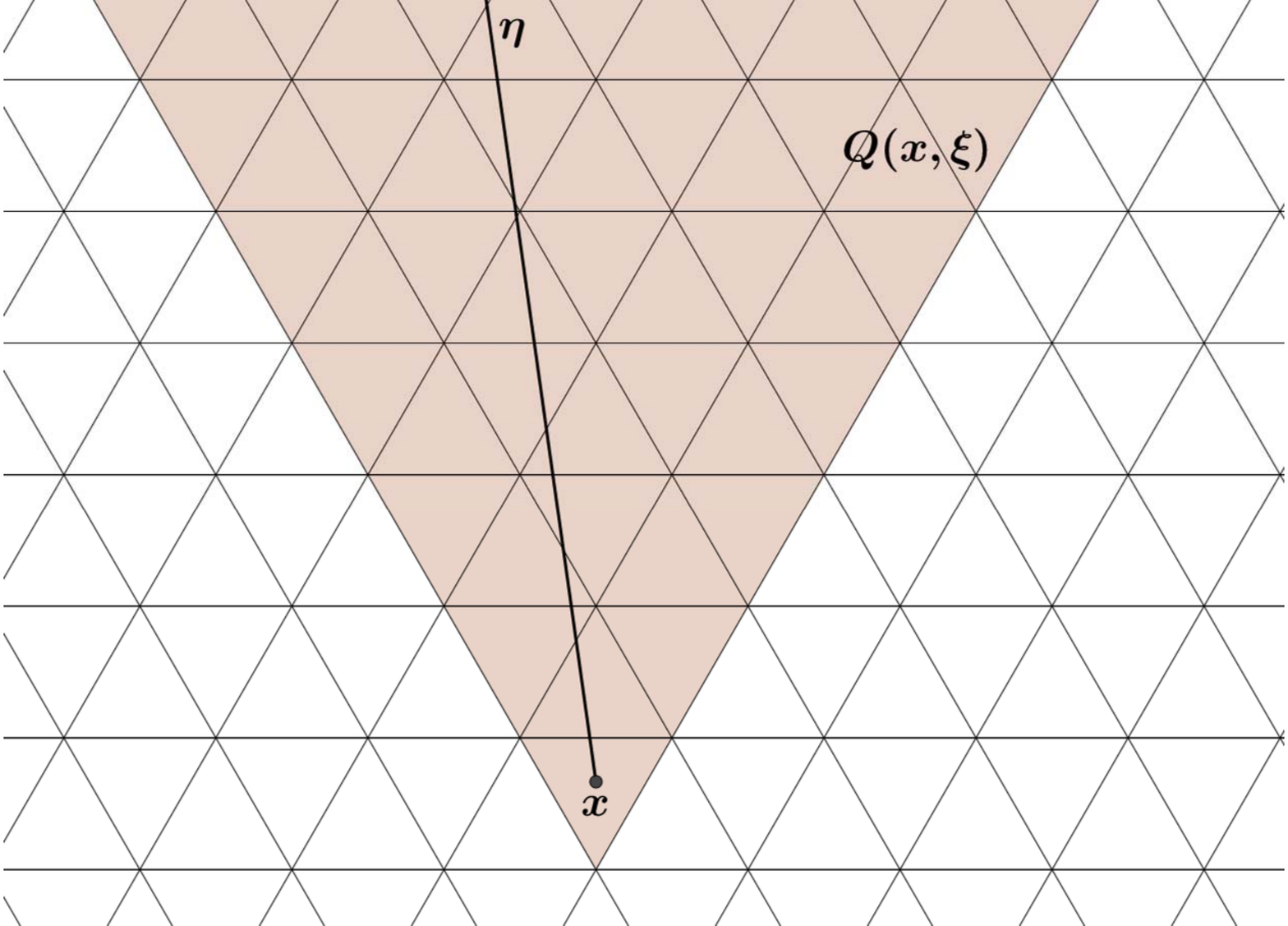}
  \captionof{figure}{Non-singular direction}
  \label{figure:nonsingular}
\end{figure}

To see what combinatorial sectors look like, we relate them to the notions introduced in \S\ref{subsection:Davis_realisation}. As in that paragraph, we identify $\Delta$ with its Davis realisation $X$. To avoid cumbersome notations, we also identify the chambers of $\Delta$ with their barycenters in $X$ (i.e. $\Ch(\Delta)$ with $X^{(0)}$): this thus also identifies the notions of minimal (resp. infinite) gallery and of combinatorial geodesic (resp. ray). For each $\eta\in\partial X$ and each apartment $A$ of $X$, we let $\WW^A_{\eta}$ denote the set of walls $m$ of $A$ containing $\eta$ in their boundary (i.e. $m$ contains a geodesic ray towards $\eta$). We also let $\mathcal{A}_{\eta}$ be the set of apartments of $X$ with $\eta\in\partial A$.

For $\eta\in\partial X$, we next define an equivalence relation $\sim_{\eta}$ on $X^{(0)}$ as follows. For $x,y\in X^{(0)}$ distinct adjacent chambers, we write $x\approx_{\eta}y$ if, for any apartment $A$ containing $x$ and $y$, the wall of $A$ separating $x$ from $y$ does not belong to $\WW_{\eta}^A$. We also write $x\approx_{\eta}x$ for $x\in X^{(0)}$, so that $\approx_{\eta}$ becomes a symmetric and reflexive relation on $X^{(0)}$. We then let $\sim_{\eta}$ be the transitive closure of $\approx_{\eta}$. For any $x\in X^{(0)}$, we now let
$$\Phi_{\eta}(x)\subseteq X$$
be the subcomplex of $X$ obtained as the union of all chambers $y\in X^{(0)}$ with $y\sim_{\eta}x$. Note that 
$$x\sim_{\eta}y\iff y\in\Phi_{\eta}(x)\iff \Phi_{\eta}(y)=\Phi_{\eta}(x)\quad\textrm{for any $y\in X^{(0)}$.}$$

We start by making some useful observations about the relation $\sim_{\eta}$.

\begin{lemma}\label{lemma:basic_approx_eta}
Let $x,y\in X^{(0)}$ and $\eta\in\partial X$, and assume that there exists an apartment of $\AAA_{\eta}$ containing $x,y$. Then $x\approx_{\eta}y$ if and only if there exists an apartment $A\in\AAA_{\eta}$ containing $x,y$ such that the wall of $A$ separating $x$ from $y$ does not belong to $\WW_{\eta}^A$.
\end{lemma}
\begin{proof}
The implication $\Rightarrow$ is clear. Conversely, let $A\in\AAA_{\eta}$ be an apartment containing $x,y$ and such that the wall $m$ of $A$ separating $x$ from $y$ does not belong to $\WW_{\eta}^A$, and assume for a contradiction that there is an apartment $A'\in\AAA_{\eta}$ containing $x,y$ such that the wall $m'$ of $A'$ separating $x$ from $y$ belongs to $\WW_{\eta}^{A'}$. Let $z\in A\cap A'$ be the barycenter of the common panel of $x$ and $y$. Then $[z,\eta)\subseteq A\cap A'$. By definition of buildings, there is a simplicial isomorphism $\phi\co A'\to A$ fixing $A\cap A'$ pointwise. Since $[z,\eta)\subseteq m'$ by assumption and $\phi(m')=m$, we deduce that $[z,\eta)=\phi([z,\eta))\subseteq m$. Hence $m\in\WW^{A}_{\eta}$, a contradiction.
\end{proof}

The next lemma introduces some important terminology and notations. For $x\in X^{(0)}$ and $\eta\in\partial X$, we call a CGR $\Gamma\in\CGR(x,\eta)$ {\bf straight} if the infinite gallery corresponding to $\Gamma$ contains the geodesic ray $[x,\eta)$.

\begin{lemma}\label{lemma:correction}
Let $x\in X^{(0)}$ and $\eta\in\partial X$. Then the following assertions hold:
\begin{enumerate}
\item
Let $\Gamma=(x_n)_{n\in\NN}\in\CGR(x,\eta)$. Then the sequence of chambers $(x_n)_{n\in\NN}$ converges in $\CCC_1(X)$. We denote its limit by $\xi_{\Gamma}\in\CCC_1(X)\setminus X^{(0)}$ and we say that $\Gamma$ {\bf converges} to $\xi_{\Gamma}$.
\item
Let $y\in X^{(0)}$, and let $\Gamma_x\in\CGR(x,\eta)$ and $\Gamma_y\in\CGR(y,\eta)$ be contained in some apartment $A$. Then $\xi_{\Gamma_x}=\xi_{\Gamma_y}$ if and only if $\Gamma_x$ and $\Gamma_y$ eventually lie on the same side of any given wall $m\in\WW_ {\eta}^{A}$.
\item
If $\Gamma\in\CGR(x,\eta)$ is straight, it is contained in $\Phi_{\eta}(x)$ and in every apartment $A\in \mathcal{A}_{\eta}$ with $x\in A^{(0)}$. Moreover, $\xi_{x,\eta}:=\xi_{\Gamma}\in\CCC_1(X)\setminus X^{(0)}$ is independent of the choice of a straight $\Gamma\in\CGR(x,\eta)$.
\end{enumerate}
\end{lemma}
\begin{proof}
(1) Since the CGR $\Gamma$ is contained in some apartment $A$, this readily follows from the above description of convergence in $\CCC_1(A)$. 

(2) We have to show that if $\Gamma_x$ and $\Gamma_y$ eventually lie on different sides of a wall $m$ of $A$, then $m\in\WW_{\eta}^A$. But if $\Gamma'_x\subseteq\Gamma_x$ and $\Gamma'_y\subseteq\Gamma_y$ are CGRs separated by $m$, then $[x,\eta)$ (which is contained in a tubular neighbourhood of $\Gamma'_x$ and $\Gamma'_y$) must be contained in a tubular neighbourhood of $m$, as claimed.

(3) Let $\Gamma\in\CGR(x,\eta)$ be straight, and let $A\in \mathcal{A}_{\eta}$ with $x\in A^{(0)}$. Thus $A$ also contains $[x,\eta)$. Since $[x,\eta)$ is not contained in any wall of $A$, we deduce that $A$ must contain infinitely many chambers of $\Gamma$, and hence also $\Gamma$ by convexity. Moreover, since $[x,\eta)$ does not intersect any wall in $\WW^A_{\eta}$, the CGR $\Gamma$ does not cross any wall in $\WW^A_{\eta}$, and hence $\Gamma\subseteq\Phi_{\eta}(x)$ by Lemma~\ref{lemma:basic_approx_eta}. Finally, if $\Gamma'\in\CGR(x,\eta)$ is straight, then both $\Gamma$ and $\Gamma'$ are contained in a common apartment $A\in \mathcal{A}_{\eta}$ by the above discussion, and hence $\xi_{\Gamma}=\xi_{\Gamma'}$ by (2).
\end{proof}

We next give an alternative description of the sets $\Phi_{\eta}(x)$.
\begin{prop}\label{prop:alt_desc_Phi_eta}
Let $x\in X^{(0)}$ and $\eta\in\partial X$. Then $$\Phi_{\eta}(x)\cap X^{(0)}=\{y\in X^{(0)} \ | \ \xi_{x,\eta}=\xi_{y,\eta}\}.$$
\end{prop}
\begin{proof}
Let $y\in X^{(0)}$. We have to show that $x\sim_{\eta}y$ if and only if $\xi_{x,\eta}=\xi_{y,\eta}$. Assume first that $x\sim_{\eta}y$. Reasoning inductively on the length of a gallery $(x=x_0,x_1,\dots,x_n=y)$ from $x$ to $y$ such that $x_{i-1}\approx_{\eta}x_i$ for all $i\in\{1,\dots,n\}$, there is no loss of generality in assuming that $x\approx_{\eta}y$. Let $\Gamma_x=(x_n)_{n\in\NN}\in\CGR(x,\eta)$ be straight. By Lemma~\ref{lemma:preparation}, there is some $k\in\NN$ and some combinatorial geodesic $\Gamma_{yx_k}$ from $y$ to $x_k$ such that $\Gamma_{yx_k\eta}:=\Gamma_{yx_k}\cdot (x_n)_{n\geq k}\in\CGR(y,\eta)$. Let $A\in\AAA_{\eta}$ be an apartment containing $\Gamma_{yx_k\eta}$. Let $\Gamma_y\in\CGR(y,\eta)$ be straight, so that $\Gamma_y\subseteq A$ by Lemma~\ref{lemma:correction}(3).
We claim that $\Gamma_{yx_k}$ does not cross any wall in $\WW^A_{\eta}$: this will imply that $\Gamma_{yx_k\eta}$ and $\Gamma_y$ do not cross any wall in $\WW^A_{\eta}$, and hence that $$\xi_{y,\eta}=\xi_{\Gamma_y}=\xi_{\Gamma_{yx_k\eta}}=\xi_{\Gamma_x}=\xi_{x,\eta}$$
by Lemma~\ref{lemma:correction}(2), as desired. Indeed, if $m\in\WW^A_{\eta}$ separates $y$ from $x_k$ and if $A'\in \AAA_{\eta}$ is an apartment containing $\Gamma_x$, then the wall $m':=\rho_{A',x_k}(m)$ of $A'$ belongs to $\WW^{A'}_{\eta}$ (because $\rho_{A',x_k}$ does not increase the distance, and hence $[x_k,\eta)=\rho_{A',x_k}([x_k,\eta))\subseteq A\cap A'$ is contained in a tubular neighbourhood of both $m$ and $m'$) and separates $x_k$ from $y':=\rho_{A',x_k}(y)$. But since $m'$ cannot be crossed by $(x_n)_{0\leq n\leq k}$, it must separate the adjacent chambers $x=\rho_{A',x_k}(x)$ and $y'$. Now, if $\alpha$ is the half-apartment of $A'$ containing $x$ and delimited by $m'$, we find by \cite[Exercise~5.83(a)]{BrownAbr} an apartment $A''$ containing $\alpha$ and $y$. Then $m'\in\WW_{\eta}^{A''}$ separates $x$ from $y$ in $A''$, contradicting our hypothesis $x\approx_{\eta}y$.

Conversely, assume that $\xi_{x,\eta}=\xi_{y,\eta}$, and let us show that $x\sim_{\eta}y$. Let $\Gamma_x=(x_n)_{n\in\NN}\in\CGR(x,\eta)$ and $\Gamma_y\in\CGR(y,\eta)$ be straight. By Lemma~\ref{lemma:preparation}, there is some $k\in\NN$ and some combinatorial geodesic $\Gamma_{yx_k}$ from $y$ to $x_k$ such that $\Gamma_{yx_k\eta}:=\Gamma_{yx_k}\cdot (x_n)_{n\geq k}\in\CGR(y,\eta)$. Let $A\in\AAA_{\eta}$ be an apartment containing $\Gamma_{yx_k\eta}$. Then $\Gamma_y\subseteq A$ by Lemma~\ref{lemma:correction}(3). Note that the walls of $A$ separating $x_k$ from $y$ do not belong to $\WW_{\eta}^A$, for otherwise the CGRs $(x_n)_{n\geq k}$ and $\Gamma_y$ (which do not cross any wall in $\WW_{\eta}^A$) would be separated by some wall of $\WW_{\eta}^A$, contradicting our assumption that $\xi_{\Gamma_x}=\xi_{x,\eta}=\xi_{y,\eta}=\xi_{\Gamma_y}$. Hence $x\sim_{\eta}x_k\sim_{\eta}y$ by Lemma~\ref{lemma:basic_approx_eta}, yielding the claim.
\end{proof}

We conclude our round of observations about the relation $\sim_{\eta}$ with the following consequences of Proposition~\ref{prop:alt_desc_Phi_eta}.

\begin{lemma}\label{lemma:correctionbis}
Let $x\in X^{(0)}$ and $\eta\in\partial X$. Then the following assertions hold:
\begin{enumerate}
\item
Let $A\in\AAA_{\eta}$ be an apartment containing $x$ and let $y\in A^{(0)}$. Then $x\sim_{\eta}y$ if and only if the walls of $A$ separating $x$ from $y$ do not belong to $\WW_{\eta}^A$.
\item
Let $\Gamma\in\CGR(x,\eta)$ and let $A$ be an apartment containing $\Gamma$. Then $\Gamma$ converges to $\xi_{x,\eta}$ $\iff$ the walls of $A$ crossed by $\Gamma$ do not belong to $\WW^A_{\eta}$ $\iff$ $\Gamma\subseteq\Phi_{\eta}(x)$.
\end{enumerate}
\end{lemma}
\begin{proof}
(1) The implication $\Leftarrow$ follows from Lemma~\ref{lemma:basic_approx_eta}. Conversely, assume that $x\sim_{\eta}y$. Then $\xi_{x,\eta}=\xi_{y,\eta}$ by Proposition~\ref{prop:alt_desc_Phi_eta}, that is, $\xi_{\Gamma_x}=\xi_{\Gamma_y}$ for some straight $\Gamma_x\in\CGR(x,\eta)$ and $\Gamma_y\in\CGR(y,\eta)$. By Lemma~\ref{lemma:correction}(3), $\Gamma_x$ and $\Gamma_y$ are contained in $A$. If now $m\in\WW_{\eta}^A$ separates $x$ from $y$ then it also separates $\Gamma_x$ from $\Gamma_y$ and hence $\xi_{\Gamma_x}\neq\xi_{\Gamma_y}$, a contradiction.

(2) This readilfy follows from (1).
\end{proof}

To get a better understanding of combinatorial sectors, we first show that, given an element $\xi\in\CCC_1(X)\setminus X^{(0)}$, one can choose a sequence of chambers $(x_n)_{n\in\NN}$ of $X$ converging to $\xi$ in a nice and controlled way. Here, by ``nice'' we mean that $\Gamma:=(x_n)_{n\in\NN}$ can be chosen to be a CGR, and by ``controlled'' we mean that we may impose further restrictions on $\Gamma$.

\begin{lemma}\label{lemma:description_CCC1X}
Let $\xi\in\CCC_1(X)\setminus X^{(0)}$. Then there is some $\eta\in\partial X$ and some straight $\Gamma=(y_n)_{n\in\NN}\in\CGR(y_0,\eta)$ such that $\xi_{\Gamma}=\xi$.
\end{lemma}
\begin{proof}
Let $A$ be an apartment of $X$ with $\xi\in\CCC_1(A)$. Let $(x_n)_{n\in\NN}$ be a sequence of chambers of $A$ converging to $\xi$. Since the space $A$ is proper, the sequence of geodesic segments $([x_0,x_n])_{n\in\NN}$ subconverges to some geodesic ray $r_0:=[x_0,\eta)$ for some $\eta\in\partial A$. In other words, up to extracting a subsequence, we may assume that $(x_n)_{n\in\NN}$ is contained in an $\epsilon$-neighbourhood $N_{\epsilon}(r_0)\subseteq A$ of $r_0$ for some $\epsilon>0$. 

We claim that there exists a finite subset $S\subseteq \WW^A_{\eta}$ such that for each $m\in\WW^A_{\eta}\setminus S$, the neighbourhood $N_{\epsilon}(r_0)$ is entirely contained in one of the half-spaces delimited by $m$. Indeed, any wall in $\WW^A_{\eta}$ intersecting $N_{\epsilon}(r_0)$, say in $y$, contains the geodesic ray $[y,\eta)\subseteq N_{\epsilon}(r_0)$ (see e.g. \cite[Theorem~II.2.13]{BHCAT0}), and hence a subray of $r_0$ in an $\epsilon$-neighbourhood. On the other hand, since $A$ is locally finite, there is some $N\in\NN$ such that any ball of radius $\epsilon$ in $A$ intersects at most $N$ walls. In particular, there are at most $N$ walls intersecting $N_{\epsilon}(r_0)$, whence the claim.

Recall that for any wall $m$ of $A$, the sequence $(x_n)_{n\in\NN}$ eventually remains on the same side of $m$; in particular, there is some $k\in\NN$ such that $(x_n)_{n\geq k}$ is entirely contained in some half-space associated to $m$ for each $m\in S$. Hence for any $n\geq k$, the walls separating $x_k$ from $x_n$ do not lie in $\WW^A_{\eta}$. Let $\Gamma=(y_n)_{n\in\NN}\in\CGR(x_k,\eta)$ be straight (thus $y_0=x_k$). Then by Lemma~\ref{lemma:correction}(3), we know that $\Gamma\subseteq \Phi_{\eta}(x_k)\cap A$. Since for any $n\geq k$ and any wall $m\in\WW_{\eta}^A$, the chambers $x_{n},x_k,y_{n}$ all lie on the same side of $m$, we conclude as in the proof of  Lemma~\ref{lemma:correction}(2) that $\xi_{\Gamma}=\xi$, as desired.
\end{proof}

\begin{prop}\label{prop:descriptionQ_nr1}
Let $x\in X^{(0)}$ and $\xi\in\CCC_1(X)\setminus X^{(0)}$. Then $$Q(x,\xi)=\{y\in X^{(0)} \ | \ \textrm{$y$ is on a CGR starting from $x$ and converging to $\xi$}\}.$$
\end{prop}
\begin{proof}
Let $\Gamma=(x_n)_{n\in\NN}$ be a CGR with $x_0=x$ and converging to $\xi$. To prove the inclusion $\supseteq$, we have to show that for any $k\in\NN$, the chamber $x_k$ belongs to $Q(x,\xi)$. But as $x_k\in \Conv(x,x_n)$ for every $n\geq k$, this is clear.

Conversely, let $y\in Q(x,\xi)$. By Lemmas~\ref{lemma:preparation} and \ref{lemma:description_CCC1X}, there exists a CGR $\Gamma=(x_n)_{n\in\NN}$ starting from $x$ and converging to $\xi$. Let $n\in\NN$ be such that $y\in\Conv(x,x_n)$. Since replacing the portion of $\Gamma$ between $x$ and $x_n$ by some combinatorial geodesic from $x$ to $x_n$ passing through $y$ still yields a CGR, the lemma follows.
\end{proof}

We next wish to prove a refinement of Proposition~\ref{prop:descriptionQ_nr1} by relating the combinatorial bordification to the visual boundary of $X$.

\begin{figure}
\centering
\includegraphics[trim = 10mm 20mm 0mm 0mm, clip,width=14cm]{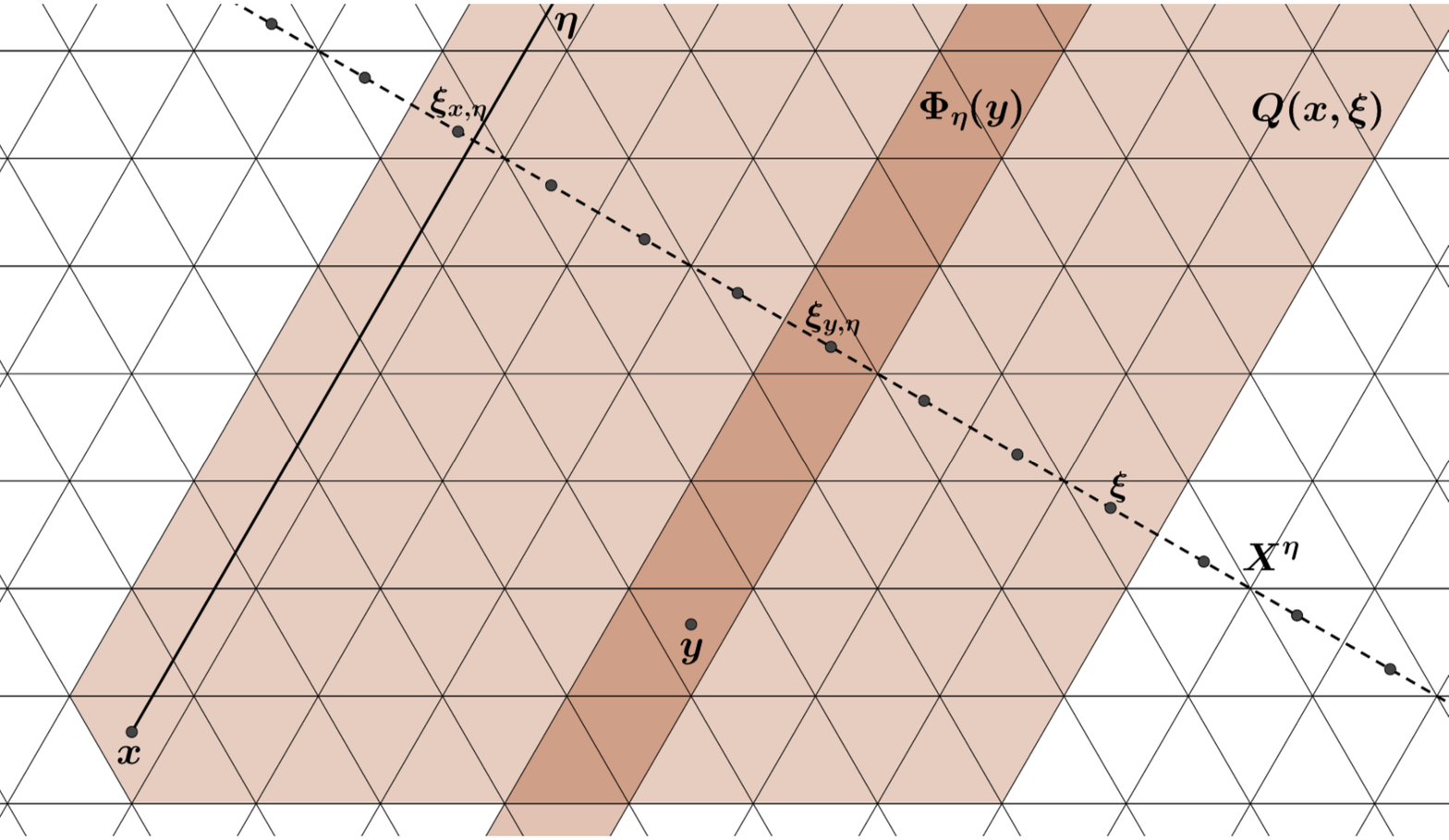}
\captionof{figure}{Singular direction}
 \label{figure:singular}
\end{figure}

For $\eta\in\partial X$, we define the {\bf transversal graph} to $X$ in the direction $\eta$ as the graph $X^{\eta}$ with vertex set
$$\CCC_{\eta}:=\{\xi_{x,\eta} \ | \ x\in X^{(0)}\}\subseteq \CCC_1(X)\setminus X^{(0)}$$
and such that $\xi_{x,\eta},\xi_{y,\eta}\in \CCC_{\eta}$ with $\xi_{x,\eta}\neq\xi_{y,\eta}$ are {\bf adjacent} (i.e. connected by an edge) if and only if there exist adjacent chambers $x',y'\in X^{(0)}$ such that $\xi_{x',\eta}=\xi_{x,\eta}$ and $\xi_{y',\eta}=\xi_{y,\eta}$. The elements of $\CCC_{\eta}$ will also be called {\bf chambers}, and we define the notions of galleries and chamber distance in $X^{\eta}$ as in \S\ref{subsection:Buildings}. Note that by Lemma~\ref{lemma:description_CCC1X}, any $\xi\in \CCC_1(X)\setminus X^{(0)}$ is of the form $\xi=\xi_{x,\eta}$ for some $x\in X^{(0)}$ and $\eta\in\partial X$, that is,
\begin{equation*}
\CCC_1(X)=X^{(0)}\cup \bigcup_{\eta\in\partial X}\CCC_{\eta}.
\end{equation*}

\begin{example}
In the context of Example~\ref{example:A21}, assume that $X$ consists of a single apartment $A$ and that $\eta\in\partial A$ is a singular direction. Then $X^{\eta}$ is a simplicial line (the dashed line on Figure~\ref{figure:singular}). The $\Phi_{\eta}(y)$ for $y\in A^{(0)}$ are stripes obtained as the convex hull of two ``adjacent'' walls of $\WW^A_{\eta}$, namely, of walls of $A$ in the direction of $\eta$.
\end{example}

Here is another description of $\CCC_{\eta}$, in terms of CGR's.

\begin{prop}\label{prop:description_Xieta}
Let $x\in X^{(0)}$ and $\eta\in\partial X$. Then
$$\CCC_{\eta}=\{\xi\in \CCC_1(X)\setminus X^{(0)} \ | \ \textrm{$\xi=\xi_{\Gamma}$ for some $\Gamma\in\CGR(x,\eta)$}\}.$$
\end{prop}
\begin{proof}
The inclusion $\subseteq$ is clear. Conversely, if $\xi=\xi_{\Gamma}$ for some $\Gamma=(x_n)_{n\in\NN}\in\CGR(x,\eta)$, then by Lemma~\ref{lemma:galleries_Phieta} below, there is some $k\in\NN$ such that $\Gamma':=(x_n)_{n\geq k}\subseteq \Phi_{\eta}(x_k)$. Hence $\xi=\xi_{\Gamma'}=\xi_{x_k,\eta}\in \CCC_{\eta}$ by Lemma~\ref{lemma:correctionbis}(2).
\end{proof}

\begin{lemma}\label{lemma:galleries_Phieta}
Let $x\in X^{(0)}$ and $\eta\in\partial X$. 
For any $\Gamma=(x_n)_{n\in\NN}\in\CGR(x,\eta)$, the sequence $(\xi_{x_n,\eta})_{n\in\NN}\subseteq \CCC_{\eta}$ is eventually constant. In other words, there is some $k\in\NN$ such that $(x_n)_{n\geq k}\subseteq \Phi_{\eta}(x_k)$.
\end{lemma}
\begin{proof}
Let $A$ be an apartment containing $\Gamma$ (thus $A\in \mathcal{A}_{\eta}$), and assume for a contradiction that $(\xi_{x_n,\eta})_{n\in\NN}\subseteq \CCC_{\eta}$ is not eventually constant. Thus there are infinitely many walls $\{m_i \ | \ i\in\NN\}\subseteq\WW^A_{\eta}$ crossed by $\Gamma$ (see Lemma~\ref{lemma:correctionbis}).

Since $[x,\eta)\subseteq A$ is not contained in any wall, it does not intersect any wall $m_i$, $i\in\NN$.
On the other hand, $\Gamma$ is contained in an $\epsilon$-neighbourhood of $[x,\eta)$ for some $\epsilon>0$. For each $y\in\Gamma$, let $y'\in [x,\eta)$ with $d(y,y')\leq\epsilon$. Then there is some $n\in\NN$ such that $[y,y']$ intersects at most $n$ walls of $A$ for any $y\in\Gamma$. Now, let $y\in\Gamma\cap A^{(0)}$ be such that the walls $m_1,\dots,m_{n+1}$ intersect the combinatorial geodesic from $x$ to $y$ contained in $\Gamma$. Then the walls $m_1,\dots,m_{n+1}$ must all intersect the geodesic segment $[y,y']$, yielding the desired contradiction.
\end{proof}

We can now formulate the announced refinement of Proposition~\ref{prop:descriptionQ_nr1}.

\begin{theorem}\label{thm:descriptionQ}
Let $x\in X^{(0)}$ and $\xi\in\CCC_{\eta}$ for some $\eta\in\partial X$. Then $$Q(x,\xi)=\{y\in X^{(0)} \ | \ \textrm{$y$ is on a CGR from $x$ to $\eta$ and converging to $\xi$}\}.$$
\end{theorem}
\begin{proof}
The inclusion $\supseteq$ follows from Proposition~\ref{prop:descriptionQ_nr1}. 
The converse inclusion is proved exactly as in the proof of Proposition~\ref{prop:descriptionQ_nr1}, the existence of a CGR from $x$ to $\eta$ and converging to $\xi$ following from Proposition~\ref{prop:description_Xieta}.
\end{proof}

Given $x\in X^{(0)}$ and $\eta\in\partial X$, we next wish to show that the combinatorial sector $Q(x,\xi_{x,\eta})$ is ``minimal'' in the direction $\eta$, in the sense that it is contained in every other combinatorial sector $Q(x,\xi)$ with $\xi\in\CCC_{\eta}$ (see Proposition~\ref{prop:Qxixinv} below). To this end, we first need a more precise version of Lemma~\ref{lemma:description_CCC1X}, by further improving our control of CGR's converging to a given $\xi\in\CCC_1(X)\setminus X^{(0)}$.

\begin{lemma}\label{lemma:galleries_Phieta2}
Let $x\in X^{(0)}$ and $\eta\in\partial X$. If $y\in X^{(0)}$ is on some CGR from $x$ to $\eta$, then $y$ is on some $\Gamma\in\CGR(x,\eta)$ converging to $\xi_{y,\eta}$.
\end{lemma}
\begin{proof}
Let $\Gamma_{xy\eta}$ be a CGR from $x$ to $\eta$ passing through $y$, and let $\Gamma_{xy}$ (resp. $\Gamma_{y\eta}$) be the combinatorial geodesic from $x$ to $y$ (resp. CGR from $y$ to $\eta$) contained in $\Gamma_{xy\eta}$, so that $\Gamma_{xy\eta}=\Gamma_{xy}\cdot\Gamma_{y\eta}$. Let $A$ be an apartment containing $\Gamma_{xy\eta}$, and let $\Gamma_{y\xi}\subseteq A$ be a straight CGR from $y$ to $\eta$ and converging to $\xi:=\xi_{y,\eta}$. In particular, $\Gamma_{y\xi}$ does not cross any wall in $\WW^A_{\eta}$. We claim that $\Gamma:=\Gamma_{xy}\cdot \Gamma_{y\xi}$ is a CGR, yielding the lemma.

Otherwise, there is a wall $m$ of $A$ that is crossed by both $\Gamma_{xy}$ and $\Gamma_{y\xi}$. Since $m$ cannot be crossed by $\Gamma_{y\eta}$, it separates $\Gamma_{y\eta}$ from $\Gamma_{y'\xi}$ for some CGR $\Gamma_{y'\xi}$ to $\eta$ contained in $\Gamma_{y\xi}$. Since $\Gamma_{y'\xi}$ and $\Gamma_{y\eta}$ are at bounded Hausdorff distance from one another, this implies that $m\in\WW^A_{\eta}$, and hence that $\Gamma_{y\xi}$ cannot cross $m$, a contradiction.
\end{proof}

\begin{prop}\label{prop:Qxixinv}
Let $x\in X^{(0)}$, $\eta\in\partial X$ and $\xi\in\CCC_{\eta}$. Then 
$$Q(x,\xi_{x,\eta})=Q(x,\xi)\cap \Phi_{\eta}(x)=\{y\in Q(x,\xi) \ | \ \xi_{y,\eta}=\xi_{x,\eta}\}.$$
\end{prop}
\begin{proof}
Note that the second equality holds by Proposition~\ref{prop:alt_desc_Phi_eta}. Let $y\in Q(x,\xi)\cap\Phi_{\eta}(x)$. Then $y$ lies on some CGR from $x$ to $\eta$ by Theorem~\ref{thm:descriptionQ}, and hence also on some $\Gamma\in\CGR(x,\eta)$ converging to $\xi_{y,\eta}$ by Lemma~\ref{lemma:galleries_Phieta2}. But since $\xi_{y,\eta}=\xi_{x,\eta}$ by Proposition~\ref{prop:alt_desc_Phi_eta}, we then have $y\in Q(x,\xi_{x,\eta})$ by Theorem~\ref{thm:descriptionQ}.

Conversely, if $y\in Q(x,\xi_{x,\eta})$ then certainly $y\in\Phi_{\eta}(x)$ by Theorem~\ref{thm:descriptionQ} and Lemma~\ref{lemma:correctionbis}(2), and it remains to show that $y\in Q(x,\xi)$. Let $A$ be an apartment containing $Q(x,\xi)$ (thus $A\in\mathcal{A}_{\eta}$). 

Note first that $A$ contains $Q(x,\xi_{x,\eta})$. Indeed, let $\Gamma=(x_n)_{n\in\NN}\in\CGR(x,\eta)$ be straight, so that $\Gamma$ converges to $\xi_{x,\eta}$ and $\Gamma\subseteq A$ by Lemma~\ref{lemma:correction}(3). Then $Q(x,\xi_{x,\eta})=\bigcup_{k\in\NN}\bigcap_{n\geq k}\Conv(x,x_n)\subseteq A$, as claimed.

Let now $\Gamma_{xy\eta}$ be a CGR from $x$ to $\eta$ passing through $y$ and converging to $\xi_{x,\eta}$ (see Theorem~\ref{thm:descriptionQ}), and let us show that $y$ is also on a CGR from $x$ to $\eta$ converging to $\xi$ (and hence that $y\in Q(x,\xi)$ by Theorem~\ref{thm:descriptionQ}, as desired).

Write $\Gamma_{xy\eta}=\Gamma_{xy}\cdot\Gamma_{y\eta}$, where $\Gamma_{xy}$ is a combinatorial geodesic from $x$ to $y$ and $\Gamma_{y\eta}$ a CGR from $y$ to $\eta$. 
By Lemma~\ref{lemma:preparation}, there is a CGR $\Gamma_{y\xi}\subseteq A$ from $y$ to $\eta$ converging to $\xi$. We claim that $\Gamma_{xy}\cdot \Gamma_{y\xi}$ is still a CGR, as desired.
Otherwise, there is a wall $m$ of $A$ that is crossed by both $\Gamma_{xy}$ and $\Gamma_{y\xi}$. Since $m$ cannot be crossed by $\Gamma_{y\eta}$, it separates $\Gamma_{y\eta}$ from $\Gamma_{y'\xi}$ for some CGR $\Gamma_{y'\xi}$ contained in $\Gamma_{y\xi}$. Since $\Gamma_{y'\xi}$ and $\Gamma_{y\eta}$ are at bounded Hausdorff distance from one another, this implies that $m\in\WW^A_{\eta}$, so that $\Gamma_{xy\eta}$ cannot cross $m$, a contradiction.
\end{proof}

To conclude this section, we give a consequence of hyperbolicity for the building $X$ in terms of the sets $\CCC_{\eta}$.

\begin{lemma}\label{lemma:hyperbolic_condition}
Assume that the building $X$ is hyperbolic. Then $\CCC_{\eta}$ is a bounded set of chambers for each $\eta\in\partial X$. In particular, if $X$ is, moreover, locally finite, then $\CCC_{\eta}$ is finite for all $\eta\in\partial X$.
\end{lemma}
\begin{proof}
Let $K$ be as in Lemma~\ref{lemma:hyperbolic_basic_prop}. Note that there exist constants $\delta_1\geq 1$ and $\delta_2\geq 0$ such that 
$$\frac{1}{\delta_1}d_{\Ch}(y,z)-\delta_2\leq d(y,z)\quad\textrm{for all $y,z\in X^{(0)}$}$$ (see \cite[Proposition~I.7.31]{BHCAT0}). We also let $\delta_3>0$ be such that for any $x\in X^{(0)}$, the closed chamber $C$ of $X$ of which $x$ is the barycenter is contained in a $\delta_3$-neighbourhood of $x$ (see \S\ref{subsection:Davis_realisation}). We fix some $N\in\NN$ such that $N\geq \delta_1(2K+\delta_2+\delta_3)$.

Let $x\in X^{(0)}$ and $\eta\in\partial X$. We claim that any $\Gamma=(x_n)_{n\in\NN}\in\CGR(x,\eta)$ crosses at most $N$ walls in $\WW^A_{\eta}$, where $A$ is any apartment containing $\Gamma$ (hence $\eta\in\partial A$). This will imply that the chambers $\{\xi_{x_n,\eta} \ | \ n\in\NN\}$ of $X^{\eta}$ are at gallery distance at most $N$ from $\xi_{x,\eta}$, and hence the lemma will follow from Proposition~\ref{prop:description_Xieta}.

Let thus $\Gamma=(x_n)_{n\in\NN}\in\CGR(x,\eta)$ and let $A$ be an apartment containing $\Gamma$. Let $\Gamma'=(y_n)_{n\in\NN}\subseteq A$ be a straight CGR from $x$ to $\eta$. By Lemma~\ref{lemma:hyperbolic_basic_prop}, we know that $\Gamma$ and $\Gamma'$ are contained in a $2K$-neighbourhood of one another. Assume for a contradiction that the combinatorial geodesic $(x_n)_{0\leq n\leq k}$ crosses $N+1$ walls in $\WW^A_{\eta}$ for some $k\in\NN$. Let $\ell\in\NN$ be such that $d(x_k,y_{\ell})\leq 2K+\delta_3$. Thus $d_{\Ch}(x_k,y_{\ell})>N$. But then
$$2K+\delta_3\leq\frac{N}{\delta_1}-\delta_2<\frac{1}{\delta_1}d_{\Ch}(x_k,y_{\ell})-\delta_2\leq d(x_k,y_{\ell})\leq 2K+\delta_3,$$ 
a contradiction.
\end{proof}

\begin{remark}
Note that, although Lemma~\ref{lemma:hyperbolic_condition} will be sufficient for our purpose, it is not hard to see (using Moussong's characterisation of hyperbolicity for Coxeter groups, see \cite[Theorem~17.1]{Mou88}) that its converse also holds: the building $X$ is hyperbolic if and only if each transversal graph $X^{\eta}$ ($\eta\in\partial X$) is bounded.
\end{remark}

\section{Geodesic ray bundles in buildings}
Throughout this section, we let $X$ be a building, identified with its Davis realisation as in Section~\ref{subsection:CBOAB}, and we keep all notations introduced in Sections~\ref{section:Preliminaries} and \ref{subsection:CBOAB}. We also fix some $\eta\in\partial X$.

We denote for each $x\in X^{(0)}$ by $\Geo(x,\eta)$ the {\bf ray bundle} from $x$ to $\eta$, that is,
$$\Geo(x,\eta):=\{y\in X^{(0)} \ | \ \textrm{$y$ lies on some $\Gamma\in\CGR(x,\eta)$}\}.$$
The description of combinatorial sectors provided in Section~\ref{subsection:CBOAB} then yields the following description of ray bundles.

\begin{prop}\label{prop:basic_description_Geo}
$\Geo(x,\eta)=\bigcup_{\xi\in \CCC_{\eta}}Q(x,\xi)$.
\end{prop}
\begin{proof}
This inclusion $\supseteq$ is clear by Theorem~\ref{thm:descriptionQ}. Conversely, if $\Gamma\in\Geo(x,\eta)$, then $\Gamma$ converges to some $\xi\in\CCC_{\eta}$ by Proposition~\ref{prop:description_Xieta}, and $\Gamma\subseteq Q(x,\xi)$ by Theorem~\ref{thm:descriptionQ}, yielding the converse inclusion.
\end{proof}

We first establish Theorem~\ref{thmintro:main} inside an apartment. 
If $x\in A^{(0)}$ for some apartment $A\in\mathcal{A}_{\eta}$, we set
$$\Geo_A(x,\eta):=\{y\in A^{(0)} \ | \ \textrm{$y$ lies on some $\Gamma\in\CGR(x,\eta)$ with $\Gamma\subseteq A$}\}.$$

\begin{lemma}\label{lemma:GeoAGeocapA}
Let $A\in\mathcal{A}_{\eta}$. Then for any $x\in A^{(0)}$,
$$\Geo_A(x,\eta)=\Geo(x,\eta)\cap A.$$
\end{lemma}
\begin{proof}
The inclusion $\subseteq$ is clear. For the converse inclusion, we have to show that if $y\in A^{(0)}$ lies on some $\Gamma\in\CGR(x,\eta)$, then it also lies on some $\Gamma'\in\CGR(x,\eta)$ with $\Gamma'\subseteq A$. But we may take $\Gamma'=\rho_{A,x}(\Gamma)$, because $\rho_{A,x}$ preserves combinatorial geodesic rays from $x$ and does not increase the distance.
\end{proof}

\begin{lemma}\label{lemma:xyGeoxGeoy}
Let $A\in\mathcal{A}_{\eta}$ and let $x\in A^{(0)}$.
If $y\in \Geo_A(x,\eta)$, then $\Geo_A(y,\eta)\subseteq\Geo_A(x,\eta)$.
\end{lemma}
\begin{proof}
Reasoning inductively on $d_{\Ch}(x,y)$, we may assume that $x$ and $y$ are adjacent. Let $m$ be the wall of $A$ separating $x$ from $y$.
Let $z\in \Geo_A(y,\eta)$ and let us show that $z\in \Geo_A(x,\eta)$. 

By Lemma~\ref{lemma:galleries_Phieta2}, we find some CGR $\Gamma_{yz\xi}\subseteq A$ from $y$ to $\eta$ going through $z$ and converging to $\xi:=\xi_{z,\eta}$. Write $\Gamma_{z\xi}$ for the CGR from $z$ to $\eta$ contained in $\Gamma_{yz\xi}$. Let also $\Gamma_{xy\eta}\subseteq A$ be a CGR from $x$ to $\eta$ going through $y$, and let $\Gamma_{y\eta}$ be the CGR from $y$ to $\eta$ contained in $\Gamma_{xy\eta}$. Finally, let $\Gamma_{xz}$ be a combinatorial geodesic from $x$ to $z$. We claim that $\Gamma_{xz}\cdot\Gamma_{z\xi}$ is a CGR, yielding the lemma.

Otherwise, there is some wall $m'$ of $A$ crossed by $\Gamma_{xz}$ and $\Gamma_{z\xi}$. Then $m'$ cannot separate $z$ from $y$ because $\Gamma_{yz\xi}$ is a CGR, and hence $m'=m$. But as $\Gamma_{xy\eta}$ is a CGR, the CGR $\Gamma_{y\eta}$ cannot cross $m$, so that $m$ separates $\Gamma_{y\eta}$ from some CGR $\Gamma_{z'\xi}$ to $\eta$ contained in $\Gamma_{z\xi}$. Since $\Gamma_{y\eta}$ and $\Gamma_{z'\xi}$ are at bounded Hausdorff distance, we deduce that $m\in\WW^A_{\eta}$, and hence that $\Gamma_{z\xi}$ cannot cross $m$, a contradiction.
\end{proof}

\begin{lemma}\label{lemma:prop_in_apt}
Assume that $X$ is hyperbolic. Let $A\in\mathcal{A}_{\eta}$ and let $x,y\in A^{(0)}$. Then the symmetric difference of $\Geo_A(x,\eta)$ and $\Geo_A(y,\eta)$ is finite.
\end{lemma}
\begin{proof}
Reasoning inductively on $d_{\Ch}(x,y)$, there is no loss of generality in assuming that $x$ and $y$ are adjacent. Assume for a contradiction that there is an infinite sequence $(y_n)_{n\in\NN}\subseteq \Geo_A(y,\eta)\setminus \Geo_A(x,\eta)$. Choose for each $n\in\NN$ some $\Gamma_n\in\CGR(y,\eta)$ passing through $y_n$. Note that if $\Gamma_{n,\leq y_n}$ denotes the combinatorial geodesic from $y$ to $y_n$ contained in $\Gamma_n$, then $\Gamma_{n,\leq y_n}$ is disjoint from $\Geo_A(x,\eta)$, for if $x'\in\Geo_A(x,\eta)\cap \Gamma_{n,\leq y_n}$, then $y_n\in\Geo_A(x',\eta)\subseteq\Geo_A(x,\eta)$ by Lemma~\ref{lemma:xyGeoxGeoy}, a contradiction.

Since $A$ is locally finite, the sequence $(\Gamma_n)_{n\in\NN}$ then subconverges to a CGR $\Gamma=(x_n)_{n\in\NN}\subseteq A$ that is disjoint from $\Geo_A(x,\eta)$. On the other hand, since $A$ is hyperbolic, Lemma~\ref{lemma:hyperbolic_basic_prop} yields that $\Gamma\in\CGR(y,\eta)$. But then Lemma~\ref{lemma:preparation} implies that $x_n\in\Geo_A(x,\eta)$ for all large enough $n$, a contradiction.
\end{proof}

We now turn to the proof of Theorem~\ref{thmintro:main} in the building $X$. For the rest of this section, we assume that $X$ is hyperbolic and locally finite, so that $\CCC_{\eta}$ is finite by Lemma~\ref{lemma:hyperbolic_condition}.

\begin{lemma}\label{lemma:z1zr}
Let $x\in X^{(0)}$ and $\xi\in\CCC_{\eta}$, and let $A$ be an apartment containing $Q(x,\xi)$. Let $S$ be an infinite subset of $Q(x,\xi)$. Then there is some $z\in Q(x,\xi)$ such that $Q(z,\xi_{z,\eta})\cap S$ is infinite.
\end{lemma}
\begin{proof}
Since $\CCC_{\eta}$ is finite, there is an infinite subset $S_1$ of $S$ and some $\xi_1\in\CCC_{\eta}$ such that $\xi_{y,\eta}=\xi_1$ for all $y\in S_1$. Let $z\in S_1$. By Lemma~\ref{lemma:prop_in_apt}, we know that $Q(x,\xi)\setminus \Geo_A(z,\eta)$ is finite, and hence by Proposition~\ref{prop:basic_description_Geo}, there is some infinite subset $S_2$ of $S_1$ contained in $Q(z,\xi_2)$ for some $\xi_2\in\CCC_{\eta}$. But then Proposition~\ref{prop:Qxixinv} implies that $S_2\subseteq Q(z,\xi_{z,\eta})$ since $\xi_{z,\eta}=\xi_1$, as desired.
\end{proof}

\begin{lemma}\label{lemma:yxiycase}
Let $x,y\in X^{(0)}$. Then $Q(y,\xi_{y,\eta})\setminus\Geo(x,\eta)$ is finite.
\end{lemma}
\begin{proof}
Assume for a contradiction that there exists an infinite sequence $(y_n)_{n\in\NN}\subseteq Q(y,\xi_{y,\eta})\setminus\Geo(x,\eta)$. By \cite[Proposition~2.30]{CL11}, there is some $z\in X^{(0)}$ such that $Q(z,\xi_{y,\eta})\subseteq Q(x,\xi_{y,\eta})\cap Q(y,\xi_{y,\eta})$ (note that this amounts to say that $Q(x,\xi_{y,\eta})\cap Q(y,\xi_{y,\eta})$ is nonempty, which readily follows from Proposition~\ref{prop:descriptionQ_nr1} together with Lemma~\ref{lemma:preparation}). Let $A$ be an apartment containing $Q(y,\xi_{y,\eta})$. 

By Lemma~\ref{lemma:prop_in_apt}, we know that $Q(y,\xi_{y,\eta})\setminus \Geo_A(z,\eta)$ is finite, and we may thus assume, up to taking a subsequence, that $(y_n)_{n\in\NN}\subseteq \Geo_A(z,\eta)$. Hence by Proposition~\ref{prop:basic_description_Geo}, we may further assume, again up to extracting a subsequence, that $(y_n)_{n\in\NN}\subseteq Q(z,\xi)$ for some $\xi\in\CCC_{\eta}$. Since $\xi_{z,\eta}=\xi_{y,\eta}$ and $(y_n)_{n\in\NN}\subseteq\Phi_{\eta}(y)$ by Proposition~\ref{prop:Qxixinv}, this proposition implies that $(y_n)_{n\in\NN}\subseteq Q(z,\xi_{z,\eta})$. But $Q(z,\xi_{z,\eta})\subseteq Q(x,\xi_{z,\eta})$ because $z\in Q(x,\xi_{y,\eta})=Q(x,\xi_{z,\eta})$, yielding the desired contradiction as $Q(x,\xi_{z,\eta})\subseteq \Geo(x,\eta)$ by Proposition~\ref{prop:basic_description_Geo}.
\end{proof}

\begin{theorem}
Let $x,y\in X^{(0)}$. Assume that $X$ is hyperbolic and locally finite. Then $\Geo(y,\eta)\setminus\Geo(x,\eta)$ is finite.
\end{theorem}
\begin{proof}
By Lemma~\ref{lemma:hyperbolic_condition}, we know that $\CCC_{\eta}$ is finite. By Proposition~\ref{prop:basic_description_Geo}, we have to prove that if $\xi\in\CCC_{\eta}$, then $Q(y,\xi)\setminus\Geo(x,\eta)$ is finite. Assume for a contradiction that there is an infinite sequence $(y_n)_{n\in\NN}\subseteq Q(y,\xi)\setminus\Geo(x,\eta)$. Then by Lemma~\ref{lemma:z1zr}, up to extracting a subsequence, we may assume that $(y_n)_{n\in\NN}\subseteq Q(z,\xi_{z,\eta})$ for some $z\in Q(y,\xi)$. This contradicts Lemma~\ref{lemma:yxiycase}.
\end{proof}

\appendix
\section{Transversal buildings}
Let $X$ be a building and let $\eta\in\partial X$. In \cite[Section~5]{CL11}, a construction of ``transversal building'' $X^{\eta}$ to $X$ in the direction $\eta$ is given. However, as pointed out to us by the referee, the premises of that construction are incorrect; the correct construction of $X^{\eta}$ is the one given in \S\ref{subsection:CBOAB} (where we called $X^{\eta}$ the ``transversal graph'' to $X$ in the direction $\eta$). Although we did not need this fact in our proof of Theorem~\ref{thmintro:main}, one can show that this transversal graph $X^{\eta}$ is indeed the chamber graph of a building, and therefore deserves the name of {\bf transversal building} to $X$ in the direction $\eta$. Since this fact is used in other papers, we devote the present appendix to its proof. Here, we will follow the ``$W$-metric approach'' to buildings (as opposed to the ``simplicial approach'' from \S\ref{subsection:Buildings}); a standard reference for this topic is \cite[\S 5.1]{BrownAbr}.

Let $(W,S)$ be the type of $X$. Let $A\in\AAA_{\eta}$, and view $W$ as a reflection group acting on $A$. Let $W_{\eta}$ be the reflection subgroup of $W$ generated by the reflections across the walls in $\WW^A_{\eta}$. By a classical result of Deodhar (\cite{Deo89}), $W_{\eta}$ is then itself a Coxeter group. Moreover, the polyhedral structure on $A$ induced by the walls in $\WW^A_{\eta}$ can be identified with the Coxeter complex of $W_{\eta}$. More precisely, if $x\in A^{(0)}$ is the fundamental chamber of $A$ (i.e. the chamber whose walls are associated to the reflections in $S$), then $x_{\eta}:=\Phi_{\eta}(x)\cap A$ (which coincides with the intersection of all half-spaces containing $x$ and whose wall belong to $\WW_{\eta}^A$) is the fundamental chamber of the Coxeter complex associated to the Coxeter system $(W_{\eta},S_{\eta})$, where $S_{\eta}$ is the set of reflections across the walls in $\WW_{\eta}^A$ that delimit $x_{\eta}$.

\begin{lemma}\label{lemma:indep_A}
The group $W_{\eta}$ depends only on $\eta$, and not on the choice of apartment $A$.
\end{lemma}
\begin{proof}
This is \cite[Lemma~5.2]{CL11} (and the proof of this lemma in \cite[\S 5.1]{CL11} remains valid in our context).
\end{proof}

\begin{lemma}\label{lemma:2chambers_in_apt}
Let $\xi,\xi'\in\CCC_{\eta}$, and let $x\in X^{(0)}$ with $\xi=\xi_{x,\eta}$. Then there exists some $y\in X^{(0)}$ with $\xi'=\xi_{y,\eta}$ such that $x$ and $y$ are contained in some apartment $A\in\AAA_{\eta}$.
\end{lemma}
\begin{proof}
Let $y\in X^{(0)}$ with $\xi'=\xi_{y,\eta}$, and let $\Gamma_y=(y_n)_{n\in\NN}\in\CGR(y,\eta)$ be straight.  By Lemma~\ref{lemma:preparation}, there is some $k\in\NN$ such that $\Gamma_{xy_k\eta}:=\Gamma_{xy_k}\cdot (y_n)_{n\geq k}\in\CGR(x,\eta)$ for some combinatorial geodesic $\Gamma_{xy_k}$ from $x$ to $y_k$. Let $A\in\AAA_{\eta}$ be an apartment containing $\Gamma_{xy_k\eta}$. Then $\xi_{y,\eta} =\xi_{y_k,\eta}$, so that the claim follows by replacing $y$ with $y_k$.
\end{proof}

\begin{theorem}
The transversal graph $X^{\eta}$ to $X$ in the direction $\eta$ is the graph of chambers of a building of type $(W_{\eta},S_{\eta})$.
\end{theorem}
\begin{proof}
We define a ``Weyl distance'' function $\delta_{\eta}\co\CCC_{\eta}\times\CCC_{\eta}\to W_{\eta}$ as follows. Let $\xi,\xi'\in\CCC_{\eta}$. By Lemma~\ref{lemma:2chambers_in_apt}, we can write $\xi=\xi_{x,\eta}$ and $\xi'=\xi_{y,\eta}$ for some $x,y\in A^{(0)}$ where $A\in\AAA_{\eta}$. We then set
$$\delta_{\eta}(\xi,\xi'):=\delta_{\eta}^A(x_{\eta},y_{\eta}),$$
where $x_{\eta}:=\Phi_{\eta}(x)\cap A$ and $y_{\eta}:=\Phi_{\eta}(x)\cap A$ are chambers in the Coxeter complex of $(W_{\eta},S_{\eta})$ and $\delta_{\eta}^A$ is the Weyl distance function on that complex (see \cite[\S 5.1]{BrownAbr}). Note that this definition is independent of the choice of apartment $A$ (see Lemma~\ref{lemma:indep_A} and its proof) and of chambers $x,y\in A^{(0)}$ such that $\xi=\xi_{x,\eta}$ and $\xi'=\xi_{y,\eta}$. To simplify the notations, we will also simply write $\delta_{\eta}^A(x,y):=\delta_{\eta}^A(x_{\eta},y_{\eta})$.

To check that $X^{\eta}$ is a building of type $(W_{\eta},S_{\eta})$, it then remains to check the axioms (WD1), (WD2) and (WD3) of \cite[Definition~5.1]{BrownAbr}. The axioms (WD1) and (WD3) are clearly satisfied, because they are satisfied in the building $(A,\delta_{\eta}^A)$ for any apartment $A\in \AAA_{\eta}$. We now check (WD2).

Let $w\in W_{\eta}$, $s\in S_{\eta}$ and $\xi_1,\xi_2,\xi_1'\in\CCC_{\eta}$ with $\delta_{\eta}(\xi_1,\xi_2)=w$ and $\delta_{\eta}(\xi_1',\xi_1)=s$. We have to show that $\delta_{\eta}(\xi_1',\xi_2)\in\{sw,w\}$ and that $\delta_{\eta}(\xi_1',\xi_2)=sw$ if $\ell_{\eta}(sw)=\ell_{\eta}(w)+1$, where $\ell_{\eta}\co W_{\eta}\to\NN$ is the length function on $W_{\eta}$ with respect to the generating set $S_{\eta}$. Choose some adjacent chambers $x,x'\in X^{(0)}$ such that $\xi_1=\xi_{x,\eta}$ and $\xi_1'=\xi_{x',\eta}$. Let also $y\in X^{(0)}$ be such that $\xi_{2}=\xi_{y,\eta}$ and such that there is an apartment $A\in\AAA_{\eta}$ with $x,y\in A^{(0)}$ (see Lemma~\ref{lemma:2chambers_in_apt}). Let $m$ be the wall of $A$ containing the $s$-panel of $x$. Since $x$ and $x'$ are separated by a wall $m'\in\WW_{\eta}^{A'}$ in some apartment $A'\in\AAA_{\eta}$ containing $x,x'$, the wall $m$ belongs to $\WW_{\eta}^A$ (as it is the image of $m'$ by the retraction $\rho_{A,x}|_{A'}$, which fixes the geodesic ray $[x,\eta)\subseteq A\cap A'$ pointwise). On the other hand, by \cite[Exercise~5.83(a)]{BrownAbr}, there is an apartment $A''\in\AAA_{\eta}$ containing $x'$ and the half-space $\alpha$ of $A$ delimited by $m$ and that contains $y$. 

If $\ell_{\eta}(sw)=\ell_{\eta}(w)+1$, then $x\in\alpha$, and hence $x,x',y\in A''$, so that 
$$\delta_{\eta}(\xi_1',\xi_2)=\delta_{\eta}^{A''}(x',y)=\delta_{\eta}^{A''}(x',x)\delta_{\eta}^{A''}(x,y)=\delta_{\eta}(\xi_1',\xi_1)\delta_{\eta}(\xi_1,\xi_2)=sw.$$
On the other hand, if $\ell_{\eta}(sw)=\ell_{\eta}(w)-1$, so that $x\notin\alpha$, then letting $x''\in A''$ denote the unique chamber of $A''$ that is $s$-adjacent to $x$ and contained in $\alpha$, we have $\delta_{\eta}^{A''}(x'',y)=sw$. Hence, in that case,
$$\delta_{\eta}(\xi_1',\xi_2)=\delta_{\eta}^{A''}(x',y)=\delta_{\eta}^{A''}(x',x'')\delta_{\eta}^{A''}(x'',y)=\delta_{\eta}^{A''}(x',x'')sw\in\{w,sw\},$$
as desired.
\end{proof}

\def\cprime{$'$}
\providecommand{\bysame}{\leavevmode\hbox to3em{\hrulefill}\thinspace}
\providecommand{\MR}{\relax\ifhmode\unskip\space\fi MR }
\providecommand{\MRhref}[2]{%
  \href{http://www.ams.org/mathscinet-getitem?mr=#1}{#2}
}
\providecommand{\href}[2]{#2}

\end{document}